\documentclass[11pt,twoside]{article}
\usepackage{amsmath,amsthm,latexsym,amssymb,graphicx,hyperref}
\usepackage{datetime}
\usepackage{dsfont}
\usepackage{mathrsfs}
\usepackage{color}

\newcommand{\thmx}{Theorem}

\newcommand{\lemx}{Lemma}

\newcommand{\defnx}{Definition}

\date{}

\begin{document}

\centerline{}

\centerline{\bf }

\centerline{}

\centerline{}

\centerline {\Large{\bf $P$-Regular Nearrings Characterized by}}

\centerline{}

\centerline {\Large{\bf Their Bi-ideals\footnote{This research is supported by the Group for Young Algebraists in University of Phayao (GYA), Thailand.}}}

\centerline{}

\centerline{\bf Aphisit Muangma and Aiyared Iampan\footnote{Corresponding author. Email: \texttt{aiyared.ia@up.ac.th}}}

\centerline{}

\centerline{Department of Mathematics, School of Science}

\centerline{University of Phayao, Phayao 56000, Thailand}

%%%%%%%%% Theorem-like environment %%%%%%%%%%%
%
\theoremstyle{plain} %text of this environment is typesetted in italics
\newtheorem{theorem}{Theorem}[section]
\newtheorem{lemma}[theorem]{Lemma}
\newtheorem{corollary}[theorem]{Corollary}
\newtheorem{proposition}[theorem]{Proposition}
\newtheorem{claim}[theorem]{Claim}
\theoremstyle{definition} %text of this environment is typesetted in roman letters
\newtheorem{definition}[theorem]{Definition}
\newtheorem{remark}[theorem]{Remark}
\newtheorem{example}[theorem]{Example}
\newtheorem{notation}[theorem]{Notation}
\newtheorem{assertion}[theorem]{Assertion}

\begin{abstract}
Using the idea of quasi-ideals of $P$-regular nearrings,
the concept of bi-ideals of $P$-regular nearrings is generalized,
which is an extension of the concept of quasi-ideals of $P$-regular nearrings and some interesting characterizations of bi-ideals are obtained.
As a result, we prove that every element of a bi-ideal
$B$ of a $P$-regular nearring can be represented as the sum of two elements of $P$ and $Q$.
Moreover, every element of the finite intersection $\displaystyle \bigcap_{i=1}^{n}B_{i}$ of bi-ideals of a $P$-regular distributive
nearring $N$ can be represented as the sum of two elements of $P$ and $B_{1}NB_{2}N\ldots NB_{n-1}NB_{n}$.
\end{abstract}



%%%%%%%%%%%%%%%%%%%%%%%%%%%%%%%%%%%%%%%%%%%%%%%%%%%%%%%%%%%%%%%%%%%%%%%%%%

\section{Introduction and Preliminaries}\numberwithin{equation}{section}

The notion of nearrings is first defined by Pilz \cite{Pilz} in 1977 and that of bi-ideals by Chelvam and Ganesan \cite{Tamizh} in 1987.
As we know, nearrings are a generalization of rings, and bi-ideals are a generalization of quasi-ideals and ideals in nearrings.
Many types of ideals on the algebraic structures were characterized by several authors such as:
In 1983, Yakabe \cite{Yakabea} introduced and characterized the notion of quasi-ideals of nearrings.
In 1987, Chelvam and Ganesan \cite{Tamizh} introduced and generalized the notion of quasi-ideals of nearrings which was introduced by \cite{Yakabea} to bi-ideals.
In 1989, Yakabe \cite{Yakabeb} characterized regular zero-symmetric nearrings without nonzero nilpotent elements in terms of quasi-ideals.
In 1990, Andrunakievich \cite{Andrunakievich} introduced $P$-regular rings.
In 1991, Choi \cite{choia} extended the $P$-regularity of rings which was introduced by \cite{Andrunakievich} to the $P$-regularity of nearrings.
In 2005, Kim, Jun and Yon \cite{Kim} introduced the notion of anti fuzzy ideals of near-rings and investigated some related properties.
In 2008, Abbasi and Rizvi \cite{Abbasi} studied prime ideals in near-rings.
In 2009, Zhan and B. Davvaz \cite{Zhan} introduced the concept of $(\overline{\in},\overline{\in}\vee\overline{q})$-fuzzy subnear-rings (ideals) of near-rings and obtain some of its related properties.
In 2010, Choi \cite{choib} gave some characterizations of quasi-ideals of $P$-regular nearrings and proved that every element of a quasi-ideal
$Q$ of a $P$-regular nearring can be represented as the sum of two elements of $P$ and $Q$.
In 2011, Dheena and Manivasan \cite{Dheena} gave some characterizations of quasi-ideals of $P$-regular nearrings in the same way as of Choi \cite{choib}.
In 2012, Sharma \cite{Sharma} studied the properties of intuitionistic fuzzy ideals of near ring with the help of their $(\alpha,\beta)$-cut sets.
The concept of quasi-ideals play an important role in studying the structure of nearrings.
Now, the notion of bi-ideals is an important and useful generalization of quasi-ideals of nearrings.
Therefore, we will study bi-ideals of nearrings in the same way as of quasi-ideals of nearrings which was studied by Choi \cite{choib}.

To present the main results we discuss some elementary definitions that we use later.

%---------------------------------------------------------------------------------------------------%

\begin{definition}\label{K1}\cite{Pilz}
A \textit{nearring} is a system consisting of a nonempty set $N$ together with two
binary operations on $N$ called addition and multiplication such that
\begin{itemize}
\item[(1)] $N$ together with addition is a group,
\item[(2)] $N$ together with multiplication is a semigroup, and
\item[(3)] $(a+b)c=ac+bc$ for all $a,b,c\in N$.
\end{itemize}
\end{definition}

For two nonempty subsets $A$ and $B$ of a nearring $N$, let
\begin{center}
$A+B:=\{a+b\mid a\in A$ and $b\in B\}$
\end{center}
and
\begin{center}
$AB:=\{ab\mid a\in A$ and $b\in B\}$.
\end{center}
If $A = \{a\}$, then we also write $\{a\}+B$ as $a+B$, and $\{a\}B$ as $aB$, and similarly if $B = \{b\}$.

%---------------------------------------------------------------------------------------------------%

\begin{definition}\label{K2}
 A nonempty subset $S$ of a nearring $N$ is called a \textit{left (right) $N$-subgroup} of $N$ if
\begin{itemize}
\item[(1)] $S$ together with addition is a subgroup of $N$, and
\item[(2)] $NS\subseteq S$ ($SN\subseteq S$).
\end{itemize}
\end{definition}

%---------------------------------------------------------------------------------------------------%

\begin{definition}\label{K3}
 A nonempty subset $S$ of a nearring $N$ is called an \textit{ideal} of $N$ if
\begin{itemize}
\item[(1)] $S$ together with addition is a normal subgroup of $N$,
\item[(2)] $SN\subseteq S$,
\item[(3)] $NS\subseteq S$, and
\item[(4)] $n_{1}(n_{2}+s)-n_{1}n_{2}\in S$ for all $s\in S$ and $n_{1},n_{2}\in N$.
\end{itemize}

Note that $S$ is a \textit{left ideal} of $N$ if $S$ satisfies (1), (3) and (4), and $S$ is a \textit{right ideal} of $N$ if $S$ satisfies (1) and (2).
\end{definition}

%---------------------------------------------------------------------------------------------------%

\begin{remark}
By \defnx~\ref{K3}, we have that
\begin{itemize}
\item[(1)] $S$ is a left ideal of $N$ if and only if $S$ is a normal left $N$-subgroup of $N$ and $n_{1}(n_{2}+s)-n_{1}n_{2}\in S$ for all $s\in S$ and $n_{1},n_{2}\in N$. \\[-0.76cm]
\item[(2)] $S$ is a right ideal of $N$ if and only if $S$ is a normal right $N$-subgroup of $N$.
\end{itemize}
\end{remark}

%---------------------------------------------------------------------------------------------------%

\begin{definition}\label{K5}
 A nearring $N$ is called a \textit{distributive nearring} if $a(b+c)=ab+ac$ for all $a,b,c\in N$.
\end{definition}

%---------------------------------------------------------------------------------------------------%

\begin{definition}\label{K6}
 A nonempty subset $Q$ of a nearring $N$ is called a \textit{quasi-ideal} of $N$ if
\begin{itemize}
\item[(1)] $Q$ together with addition is a subgroup of $N$, and
\item[(2)] $QN\cap NQ\subseteq Q$.
\end{itemize}
\end{definition}

%---------------------------------------------------------------------------------------------------%

\begin{definition}\label{K7}
A nonempty subset $B$ of a nearring $N$ is called a \textit{bi-ideal} of $N$ if
\begin{itemize}
\item[(1)] $B$ together with addition is a subgroup of $N$, and
\item[(2)] $BNB\subseteq B$.
\end{itemize}
\end{definition}
%---------------------------------------------------------------------------------------------------%

\begin{definition}\label{K8}
A nearring $N$ is called \textit{regular nearring} if for each $x\in N$ there exists $y\in N$ such that $xyx=x$.
\end{definition}

%---------------------------------------------------------------------------------------------------%

\begin{definition}\label{K9}
Let $N$ be a nearring with unity and $P$ an ideal of $N$. Then $N$ is said to be \textit{$P$-regular nearring}
if for each $x\in N$ there exists $y\in N$ such that $xyx-x\in P$.
\end{definition}

%%%%%%%%%%%%%%%%%%%%%%%%%%%%%%%%%%%%%%%%%%%%%%%%%%%%%%%%%%%%%%%%%%%%%%%%%%

\section{Lemmas}

Before the characterizations of bi-ideals of nearrings for the main results,
we give some auxiliary results which are necessary in what follows.

\begin{lemma}\label{K10}\cite{choib}
Let $N$ be a nearring and $P=\{0\}$. If $N$ is a $P$-regular nearring, then $N$ is a regular nearring.
\end{lemma}

%-------------------------------------------------------------------------------------------------------------------

\begin{lemma}\label{K13}
Let $\mathcal{B}$ be a nonempty family of bi-ideals of a nearring $N$. Then $\bigcap\mathcal{B}$ is a bi-ideal of $N$.
\end{lemma}

\begin{proof}
Clearly, $\bigcap\mathcal{B}$ together with addition is a subgroup of $N$.
Now, for all $B\in\mathcal{B}$, we have
\begin{center}
$\bigcap \mathcal{B} N\bigcap \mathcal{B} \subseteq BNB\subseteq B$.
\end{center}
Thus $\bigcap \mathcal{B} N\bigcap \mathcal{B} \subseteq \bigcap \mathcal{B}$. Hence $\bigcap \mathcal{B}$ is a bi-ideal of $N$.
\end{proof}

%---------------------------------------------------------------------------------------------------%

\begin{corollary}\label{K13_1}
Any finite intersection of bi-ideals of a nearring is a bi-ideal.
\end{corollary}

%---------------------------------------------------------------------------------------------------%

\begin{lemma}\label{K13_2}
Every quasi-ideal of a nearring is a bi-ideal.
\end{lemma}

\begin{proof}
Let $Q$ be a quasi-ideal of a nearring $N$. Then $Q$ together with addition is a subgroup of $N$.
Thus $QNQ\subseteq QN$ and $QNQ\subseteq NQ$, so $QNQ\subseteq QN\cap NQ\subseteq Q$.
Hence $Q$ is a bi-ideal of $N$.
\end{proof}

%%%%%%%%%%%%%%%%%%%%%%%%%%%%%%%%%%%%%%%%%%%%%%%%%%%%%%%%%%%%%%%%%%%%%%%%%%

\section{Main Results}

In this section, give some characterizations of bi-ideals of nearrings.
Finally, we prove that every element of a bi-ideal
$B$ of a $P$-regular nearring can be represented as the sum of two elements of $P$ and $Q$.
Moreover, every element of the finite intersection $\displaystyle \bigcap_{i=1}^{n}B_{i}$ of bi-ideals of a $P$-regular distributive
nearring $N$ can be represented as the sum of two elements of $P$ and $B_{1}NB_{2}N\ldots NB_{n-1}NB_{n}$.

\begin{theorem}\label{Th2.1}
Let $N$ be a $P$-regular nearring. Then for each $n\in N$ there exists $n^{\prime}\in N$ such that $n'n\in P$.
\end{theorem}

%---------------------------------------------------------------------------------------------------%

\begin{theorem}\label{Th2.2}
Let $N$ be a $P$-regular distributive nearring. Then for every right ideal $R$ and every left ideal $L$ of $N$,
\begin{center}
$( P + R )\cap( P + L ) = P + RL$.
\end{center}
\end{theorem}

%---------------------------------------------------------------------------------------------------%

\begin{theorem}\label{Th2.3}
 Let $N$ be a $P$-regular nearring and $B$ a bi-ideal of $N$. Then every $x\in B$ there exist $p^{\prime}\in P$ and $b^{\prime}\in B$ such that $x=p^{\prime}+b^{\prime}$.
\end{theorem}

\begin{proof}
Let $x\in B$.
Since $N$ is a $P$-regular nearring and $x\in B\subseteq N$, there exists $y\in N$ such that $xyx-x=p$ for some $p\in P$.
Thus $x=-p+xyx$.
Since $B$ is a bi-ideal of $N$, we have $xyx\in BNB\subseteq B$.
Since $p\in P$ and $P$ together with addition is a subgroup of $N$, we have $-p\in P$.
Put $p^{\prime}=-p$ and $b^{\prime}=xyx$. Thus
\begin{center}
$x=-p+xyx=p^{\prime}+b^{\prime}\in P+B$.
\end{center}
\end{proof}

%---------------------------------------------------------------------------------------------------%

\begin{theorem}\label{Th2.4}
Let $N$ be a $P$-regular distributive nearring and $B_{1}$ and $B_{2}$ bi-ideals of $N$.
If $b\in B_{1}\cap B_{2}$ and $x\in N$, then the element $b$ can be represented as
\begin{center}
$b=p+b_{1}x_{1}b_{2}$ and $b_{1}x_{1}b_{2}xP\subseteq P$
\end{center}
for some $p\in P,x_{1}\in N,b_{1}\in B_{1}$ and $b_{2}\in B_{2}$.
\end{theorem}

\begin{proof}
Let $b\in B_{1}\cap B_{2}$.
Since $N$ is a $P$-regular nearring, there exists $x_{1}\in N$ such that
$bx_{1}b-b\in P$.
By \lemx~\ref{K13}, we have $B_{1}\cap B_{2}$ is a bi-ideal of $N$.
Since $b\in B_{1}\cap B_{2}$, we have $b\in B_{1}$ and $b\in B_{2}$.
By \thmx~\ref{Th2.3}, we have $b=p_{1}+b_{1}$ for some $p_{1}\in P$ and $b_{1}\in B_{1}$,
and $b=p_{2}+b_{2}$ for some $p_{2}\in P$ and $b_{2}\in B_{2}$.
Since $bx_{1}b-b\in P$, we have $bx_{1}b-b=p_{3}$ for some $p_{3}\in P$.
Thus $b=-p_{3}+bx_{1}b$. Hence
\begin{eqnarray*}
b &=& -p_{3}+bx_{1}b  \\
&=& -p_{3}+(p_{1}+b_{1})x_{1}(p_{2}+b_{2})  \\
&=& -p_{3}+p_{1}x_{1}p_{2}+p_{1}x_{1}b_{2}+b_{1}x_{1}p_{2}+b_{1}x_{1}b_{2}.
 \end{eqnarray*}
Since $P$ is an ideal of $N$, we have
$-p_{3},p_{1}x_{1}p_{2},p_{1}x_{1}b_{2},b_{1}x_{1}p_{2}\in P$.
Then $-p_{3}+p_{1}x_{1}p_{2}+p_{1}x_{1}b_{2}+b_{1}x_{1}p_{2}=p_{4}$ for some $p_{4}\in P$.
Thus $b=p_{4}+b_{1}x_{1}b_{2}$, so $b_{1}x_{1}b_{2}=-p_{4}+b$.
Hence
\begin{center}
$b_{1}x_{1}b_{2}xP=(-p_{4}+b)xP\subseteq -p_{4}xP+bxP\subseteq P+P\subseteq P$.
\end{center}
\end{proof}

%---------------------------------------------------------------------------------------------------%

\begin{theorem}\label{Th2.5_1}
Let $N$ be a $P$-regular distributive nearring and $\{B_{i}\mid i\in \mathds{Z}$ and $1\leq i\leq n\}$ a nonempty family of bi-ideals of $N$.
If $\displaystyle b\in \bigcap_{i=1}^{n}B_{i}$ and $x\in N$, then the element $b$ can be represented as
\begin{center}
$b=p+b_{1}x_{1}b_{2}x_{2}\ldots b_{n-1}x_{n-1}b_{n}$ and $b_{1}x_{1}b_{2}x_{2}\ldots b_{n-1}x_{n-1}b_{n}xP\subseteq P$
\end{center}
for some $p\in P,x_{1},x_{2},\ldots,x_{n-1}\in N$ and $b_{i}\in B_{i}$ for all $1\leq i\leq n$.
\end{theorem}

\begin{proof}
If $b\in B_{1}$, then by \thmx~\ref{Th2.3}, we have $b=p+b_{1}$ for some $p\in P$ and $b_{1}\in B_{1}$.
Thus
\begin{center}
$b_{1}xP=(-p+b)xP\subseteq -pxP+bxP\subseteq P+P\subseteq P$.
\end{center}
Assume that the theorem is true for integer $n-1$.
Let $\displaystyle b\in \bigcap_{i=1}^{n}B_{i}.$
Since $\displaystyle\bigcap_{i=1}^{n}B_{i}\subseteq \bigcap_{i=1}^{n-1}B_{i}$ and $\displaystyle\bigcap_{i=1}^{n}B_{i}\subseteq B_{n}$, we have $\displaystyle b\in \bigcap_{i=1}^{n-1}B_{i}$ and $b\in B_{n}.$
By assumption, we have
\begin{equation}\label{Eq15}
b=p_{1}+b_{1}x_{1}b_{2}x_{2}\ldots b_{n-2}x_{n-2}b_{n-1}
\end{equation}
and $b_{1}x_{1}b_{2}x_{2}\ldots b_{n-2}x_{n-2}b_{n-1}xP\subseteq P$ for some $p_{1}\in P,x_{1},x_{2},\ldots ,x_{n-2}\in N$ and $b_{i}\in B_{i}$ for all $1\leq i\leq n-1$.
By \thmx~\ref{Th2.3}, we have
\begin{equation}\label{Eq16}
b=p_{2}+b_{n}
\end{equation}
for some $p_{2}\in P$ and $b_{n}\in B_{n}$.
Since $N$ is a $P$-regular nearring, there exists $x_{n-1}\in N$ such that
$bx_{n-1}b-b\in P$.
Thus $bx_{n-1}b-b=p_{3}$ for some $p_{3}\in P$, so $b=-p_{3}+bx_{n-1}b$.
By \eqref{Eq15} and \eqref{Eq16}, we have
\begin{equation}\label{Eq17}
bx_{n-1}b=(p_{1}+b_{1}x_{1}b_{2}x_{2}\ldots b_{n-2}x_{n-2}b_{n-1})x_{n-1}(p_{2}+b_{n}).
\end{equation}
By \eqref{Eq17}, we have
\begin{eqnarray*}
b &=&-p_{3}+bx_{n-1}b \\
&=& -p_{3}+(p_{1}+b_{1}x_{1}b_{2}x_{2}\ldots b_{n-2}x_{n-2}b_{n-1})x_{n-1}(p_{2}+b_{n}) \\
&=& -p_{3}+p_{1}x_{n-1}p_{2}+p_{1}x_{n-1}b_{n}+ \\
&& b_{1}x_{1}b_{2}x_{2}\ldots b_{n-2}x_{n-2}b_{n-1}x_{n-1}p_{2}+ \\
&& b_{1}x_{1}b_{2}x_{2}\ldots b_{n-2}x_{n-2}b_{n-1}x_{n-1}b_{n}.
\end{eqnarray*}
Put $-p_{3}+p_{1}x_{n-1}p_{2}+p_{1}x_{n-1}b_{n}+b_{1}x_{1}b_{2}x_{2}\ldots b_{n-2}x_{n-2}b_{n-1}x_{n-1}p_{2}=p_{4}$ for some $p_{4}\in P$. Thus
\begin{center}
$b=p_{4}+b_{1}x_{1}b_{2}x_{2}\ldots b_{n-2}x_{n-2}b_{n-1}x_{n-1}b_{n}$.
\end{center}
That is $b_{1}x_{1}b_{2}x_{2}\ldots b_{n-2}x_{n-2}b_{n-1}x_{n-1}b_{n}=-p_{4}+b$.
Hence
\begin{eqnarray*}
b_{1}x_{1}b_{2}x_{2}\ldots b_{n-2}x_{n-2}b_{n-1}x_{n-1}b_{n}xP &=& (-p_{4}+b)xP \\
&\subseteq& -p_{4}xP+bxP \\
&\subseteq& P+P \\
&\subseteq& P.
\end{eqnarray*}
\end{proof}

%---------------------------------------------------------------------------------------------------%

\begin{theorem}\label{Th2.6}
Let $N$ be a $P$-regular nearring and $B$ a bi-ideal of $N$. Then
\begin{center}
$P+B=P+BNB$.
\end{center}
\end{theorem}

\begin{proof}
Since $B$ is a bi-ideal of $N$, we have $BNB\subseteq B$. Thus
\begin{equation}\label{Eq11}
P+BNB\subseteq P+B.
\end{equation}
On the other hand, let $n\in P+B$. Then $n=p^{\prime}+b^{\prime}$ for some $p^{\prime}\in P$ and $b^{\prime}\in B$.
Since $N$ is a $P$-regular nearring, there exists $x\in N$ such that
$b^{\prime}xb^{\prime}-b^{\prime}\in P$.
Thus $b^{\prime}xb^{\prime}-b^{\prime}=p^{\prime\prime}$ for some $p^{\prime\prime}\in P$, so $b^{\prime}=-p^{\prime\prime}+b^{\prime}xb^{\prime}$.
Therefore
\begin{center}
$n=p^{\prime}+b^{\prime}=p^{\prime}+(-p^{\prime\prime}+b^{\prime}xb^{\prime})=(p^{\prime}-p^{\prime\prime})+b^{\prime}xb^{\prime}\in P+BNB$.
\end{center}
Hence
\begin{equation}\label{Eq12}
P+B\subseteq P+BNB.
\end{equation}
By \eqref{Eq11} and \eqref{Eq12}, we have $P+B=P+BNB$.
\end{proof}

%---------------------------------------------------------------------------------------------------%

\begin{theorem}\label{Th2.7}
Let $N$ be a $P$-regular nearring, and $B_{1}$ and $B_{2}$ bi-ideals of $N$. Then
\begin{center}
$P+(B_{1}\cap B_{2})\subseteq P+(B_{1}NB_{2}\cap B_{2}NB_{1})$.
\end{center}
\end{theorem}

\begin{proof}
Let $b\in P+(B_{1}\cap B_{2})$.
Then $b=p+b^{\prime}$ for some $p\in P$ and $b^{\prime}\in B_{1}\cap B_{2}$.
Thus $b^{\prime}\in B_{1}$ and $b^{\prime}\in B_{2}$.
Since $N$ is a $P$-regular nearring, there exists $x\in N$ such that
$b^{\prime}xb^{\prime}-b^{\prime}\in P$.
Thus $b^{\prime}xb^{\prime}-b^{\prime}=p^{\prime}$ for some $p^{\prime}\in P$, so
$b^{\prime}=-p^{\prime}+b^{\prime}xb^{\prime}.$
Hence
\begin{center}
$b=p+b^{\prime}=p-p^{\prime}+b^{\prime}xb^{\prime}=p^{\prime\prime}+b^{\prime}xb^{\prime}\in P+(B_{1}NB_{2}\cap B_{2}NB_{1})$
\end{center}
where $p^{\prime\prime}=p-p^{\prime}$.
Therefore
\begin{equation}\label{Eq13}
P+(B_{1}\cap B_{2})\subseteq P+(B_{1}NB_{2}\cap B_{2}NB_{1}).
\end{equation}
\end{proof}

%---------------------------------------------------------------------------------------------------%

\begin{theorem}\label{Th2.8}
Let $N$ be a $P$-regular nearring, and $\{B_{i}\mid i\in \mathds{Z}$ and $1\leq i\leq n\}$ a nonempty family of bi-ideals of $N$. Then
\begin{center}
$P+(\displaystyle \bigcap_{i=1}^{n}B_{i})\subseteq P+(B_{1}NB_{n}\cap B_{2}NB_{n}\cap\ldots\cap B_{n-1}NB_{n}\cap B_{n}NB_{1}\cap B_{n}NB_{2}\cap\ldots\cap B_{n}NB_{n-1})$.
\end{center}
\end{theorem}

\begin{proof}
By \thmx~\ref{Th2.6}, we have $P+B_{1}=P+B_{1}NB_{1}$. That is $P+B_{1}\subseteq P+B_{1}NB_{1}$.
Assume that the theorem is true for integer $n-1$.
By \thmx~\ref{Th2.7}, we have
\begin{eqnarray*}
P+(\displaystyle \bigcap_{i=1}^{n}B_{i})&=&P+(\displaystyle \bigcap_{i=1}^{n-1}B_{i}\cap B_{n})\\
&\subseteq&P+(\displaystyle (\bigcap_{i=1}^{n-1}B_{i})NB_{n}\cap B_{n}N\displaystyle (\bigcap_{i=1}^{n-1}B_{i}))\\
&\subseteq&P+((B_{1}\cap B_{2}\cap\ldots\cap B_{n-1})NB_{n}\cap B_{n}N(B_{1}\cap B_{2}\cap\ldots\cap B_{n-1}))\\
&\subseteq&P+(B_{1}NB_{n}\cap B_{2}NB_{n}\cap\ldots\cap B_{n-1}NB_{n}\cap\\
&&B_{n}NB_{1}\cap B_{n}NB_{2}\cap\ldots\cap B_{n}NB_{n-1}).
\end{eqnarray*}
\end{proof}

%%%%%%%%%%%%%%%%%%%%%%%%%%%%%%%%%%%%%%%%%%%%%%%%%%%%%%%%%%%%%%%%%%%%%%%%%%

\section*{Acknowledgment}
The authors wish to express their sincere thanks to the referees for the valuable suggestions which lead to an improvement of this paper.

%%%%%%%%%%%%%%%%%%%%%%%%%%%%%%%%%%%%%%%%%%%%%%%%%%%%%%%%%%%%%%%%%%%%%%%%%%

%\bibliographystyle{amsplain}
%\bibliography{references}

\begin{thebibliography}{10}

\bibitem{Abbasi}
S.~J. Abbasi and A.~Z. Rizvi, \emph{Study of prime ideals in
  near\textrm{-}ring}, Journal of Engineering and Sciences \textbf{2} (2008),
  65--66.

\bibitem{Andrunakievich}
V.~A. Andrunakievich, \emph{Regularity of a ring with a respect to right
  ideals}, Doklady Akademii Nauk SSSR \textbf{310} (1990), 267--272.

\bibitem{Tamizh}
T.~T. Chelvam and N.~Ganesan, \emph{On bi\textrm{-}ideals of
  near\textrm{-}rings}, Indian Journal of Pure and Applied Mathematics
  \textbf{18} (1987), 1002--1005.

\bibitem{choia}
S.~J. Choi, \emph{\textrm{P}\textrm{-}regularity of a near-ring}, Master's
  thesis, University of Dong-A, 1991.

\bibitem{choib}
S.~J. Choi, \emph{Quasideal of a \textrm{P}\textrm{-}regular near\textrm{-}ring},
  International Journal of Algebra \textbf{4} (2010), 501--506.

\bibitem{Dheena}
P.~Dheena and S.~Manivasan, \emph{Quasiideals of a \textrm{P}\textrm{-}regular
  near\textrm{-}rings}, International Journal of Algebra \textbf{5} (2011),
  1005--1010.

\bibitem{Kim}
K.~H. Kim, Y.~B. Jun, and Y.~H. Yon, \emph{On anti fuzzy ideals in near-rings},
  Iranian Journal of Fuzzy Systems \textbf{2} (2005), 71--80.

\bibitem{Pilz}
G.~Pilz, \emph{Near\textrm{-}rings}, North-Holland Publishing Company, New
  York, 1977.

\bibitem{Sharma}
P.~K. Sharma, \emph{Intuitionistic fuzzy ideals of near rings}, International
  Mathematical Forum \textbf{7} (2012), 769--776.

\bibitem{Yakabea}
I.~Yakabe, \emph{Quasi\textrm{-}ideals in near\textrm{-}rings}, Mathematical
  reports College of General Education Kyushu University \textbf{14} (1983),
  41--46.

\bibitem{Yakabeb}
I.~Yakabe, \emph{Regular near\textrm{-}rings without non\textrm{-}zero nilpotent
  elements}, Proceedings of the Japan Academy, Series A \textbf{65} (1989),
  176--179.

\bibitem{Zhan}
J.~M. Zhan and B.~Davvaz, \emph{Generalized fuzzy ideals of
  near\textrm{-}rings}, Applied Mathematics-A Journal of Chinese Universities
  \textbf{24} (2009), 343--349.

\end{thebibliography}

{\bf Received: \today}

\end{document}